\documentclass{article}

\usepackage[margin=1in]{geometry}

\usepackage{amsmath,amssymb,amsthm}
\usepackage{graphics}
\usepackage[dvips]{graphicx}
\usepackage{eepic}
\usepackage{epic}
\usepackage{enumerate}
\usepackage{mathrsfs}
\usepackage[all]{xy}

\theoremstyle{definition}

\newtheorem{thm}{Theorem}[section]

\newtheorem{lmm}[thm]{Lemma}
\newtheorem{prp}[thm]{Proposition}
\newtheorem{dfn}[thm]{Definition}

\theoremstyle{remark}
\newtheorem*{rem}{Remark}

\title{A survey of known results on  \\the $m$-step solvable anabelian geometry  for hyperbolic curves}
\author{\textsc{Naganori Yamaguchi}\footnote{RIMS, Kyoto University, Kyoto
606-8502, Japan.\newline e-mail: \texttt{naganori@kurims.kyoto-u.ac.jp}}}
\date{}

\begin{document}

\maketitle

\tableofcontents      
\begin{abstract}      

In this survey, we introduce the three theorems about  the  $m$-step solvable Grothendieck conjecture in anabelian geometry of hyperbolic curves  by H. Nakamura, S. Mochizuki, and the author. We also give sketches of the proofs of these theorems. 
\end{abstract}

\addcontentsline{toc}{section}{Introduction}
\section*{Introduction}\ \par

From now on, we fix the following notation. Let $m$  be an  integer greater than or equal to  $1$,  $k$ a field of characteristic $p\geq 0$,   $X$ (resp. $X_{1}$, $X_{2}$) a proper, smooth curve over $k$ (in other words, $X$  (resp. $X_{1}$, $X_{2}$)  is  geometrically connected, proper, smooth of  dimension one over $k$), and  $E$ (resp. $E_{1}$, $E_{2}$) a closed subscheme of $X$ (resp. $X_{1}$, $X_{2}$) which is finite, \'{e}tale  over $k$. We set $U:=X-E$ (resp. $U_{1}:=X_{1}-E_{1}$, $U_{2}:=X_{2}-E_{2}$) and call $U$ (resp. $U_{1}$, $U_{2}$) a smooth curve over $k$. We write  $g$ (resp. $g_{1}$, $g_{2}$)  for a geometric genus of $X_{\overline{k}}$ (resp. $X_{1,\overline{k}}$,  $X_{2,\overline{k}}$), and set  $r:=|E(\overline{k})|$ (resp.  $r_{1}:=|E_{1}(\overline{k})|$,  $r_{2}:=|E_{2}(\overline{k})|$). We say that the smooth curve $U$ is  $hyperbolic$ if  $2-2g-r<0$.\par
Let $K(U_{k^{\text{sep}}})$ be the function field of $U_{k^{\text{sep}}}$, $\Omega$  an algebraically closed field  containing $K(U_{k^{\text{sep}}})$,   $\overline{\eta}:\text{Spec}(\Omega)\rightarrow U_{k^{\text{sep}}}$ the corresponding  geometric point,    $\Sigma$ a set of primes that contains at least one  prime different from $p$, and $\Sigma^{\dag}:=\Sigma-\{p\}$. We set
\begin{equation*}
\Pi_{U}:=\pi_1^{\text{tame}}(U,\overline{\eta}),\ \ \ \text{and}\  \ \ \overline{\Pi}_{U}:=\pi_1^{\text{tame}}(U_{k^{\text{sep}}},\overline{\eta}).
\end{equation*}
 For a profinite  group $G$, we write $G^{\Sigma}$ for the maximal pro-$\Sigma$ quotient of $G$, and    $\overline{[G,G]} $ for the closed subgroup of $G$ which is (topologically) generated by the commutator subgroup of $G$. We set $G^{[0]}:=G$,  $G^{[m]}:=\overline{[G^{[m-1]},G^{[m-1]}]}$,  $G^{0}:=G/G^{[0]}$, and  $G^{m}:=G/G^{[m]}$.   We set    $G^{m,\Sigma}:=(G^{\Sigma})^{m}$ for simplicity. We define 
\[
\Pi^{(m)}_{U}:= \Pi_{U}/\overline{\Pi}_{U}^{[m]},\ \ \ \ \Pi_{U}^{(m,\Sigma)}:=\Pi_{U}/\text{Ker} (\overline{\Pi}_{U}\twoheadrightarrow\overline{\Pi}_{U}^{m,\Sigma}),
\]
$\Pi^{(0)}_{U}:= \Pi_{U}/\overline{\Pi}_{U}^{[0]}(\xrightarrow{\sim}G_{k})$, and $\Pi_{U}^{(0,\Sigma)}:=\Pi_{U}/\text{Ker} (\overline{\Pi}_{U}\twoheadrightarrow\overline{\Pi}_{U}^{0,\Sigma})(\xrightarrow{\sim}G_{k})$. When $\Sigma=\{\ell\}$,  we write $\overline{\Pi}_{U}^{\text{pro-}\ell}$,  $\overline{\Pi}_{U}^{m, \text{pro-}\ell}$, $\Pi^{(\text{pro-}\ell)}_{U}$, $\Pi_{U}^{(m,\text{pro-}\ell)}$   instead of $\overline{\Pi}_{U}^{\Sigma}$, $\overline{\Pi}_{U}^{m, \Sigma}$, $\Pi^{(\Sigma)}_{U}$, $\Pi_{U}^{(m,\Sigma)}$,  respectively. When $\Sigma$ is the set of  all primes different from $p$, we write  $\overline{\Pi}_{U}^{p'}$, $\overline{\Pi}_{U}^{m, p'}$, $\Pi^{(p')}_{U}$, $\Pi_{U}^{(m,p')}$ instead of $\overline{\Pi}_{U}^{\Sigma}$,  $\overline{\Pi}_{U}^{m, \Sigma}$, $\Pi^{(\Sigma)}_{U}$,  $\Pi_{U}^{(m,\Sigma)}$, respectively.   For each quotient $\Pi_{U}\twoheadrightarrow Q$ which satisfies $\overline{\Pi}_{U}\subset \text{Ker}(\Pi_{U}\twoheadrightarrow Q)$, we  write $\text{pr}:Q\twoheadrightarrow G_{k}$ for the natural projection.  \par
We write ``\textbf{FF}'',  ``\textbf{NF}'', ``\textbf{FGF}'', ``$\textbf{FGF}_{\infty}$''  for   ``finite field'',  ``number field'' ,   ``field finitely generated over the prime field'',  ``infinite field finitely generated over the prime field'', respectively.  Moreover, we write  ``\textbf{S}$\ell$\textbf{F}''  for  ``sub $\ell$-adic field'' (i.e., a subfield of a finitely generated field extension of $\mathbb{Q}_{\ell}$).\par
In anabelian geometry, we have a fundamental conjecture called  the (weak) Grothen-dieck conjecture, which  predicts:  if a $G_k$-isomorphism $\Pi_{U_{1}}\cong\Pi_{U_{2}}$ exists, a $k$-isomorphism $U_{1}\cong U_{2}$ exists (we only consider the case of $p=0$ for simplicity).  About the Grothendieck conjecture, we already have many results, e.g., \cite{Na1990-411}, \cite{Ta1997}, \cite{Mo1999}, \cite{St2002}. We can consider  a variant of the Grothendieck conjecture by  replacing $\Pi_{U_{1}}$, $\Pi_{U_{2}}$  with  $\Pi^{(m)}_{U_{1}}$, $\Pi^{(m)}_{U_{2}}$, respectively. Then we are led to the following conjecture when $p=0$. 
\[
\Pi^{(m)}_{U_{1}}\underset{G_k}{\cong} \Pi^{(m)}_{U_{2}}\iff
U_{1}\underset{k}{\cong}U_{2}
\]

\noindent  If $p>0$, we need to modify  the above conjecture because the morphism $\Pi_{U}\rightarrow \Pi_{U(1)}$ induced by the (relative) Frobenius morphism $U\rightarrow U(1)$ is an isomorphism over $G_{k}$ (\cite{Ya2020} Remark 2.3.2).  Here, we consider the following conjecture when $p>0$.
\[
\Pi^{(m)}_{U_{1}}\underset{G_k}{\cong} \Pi^{(m)}_{U_{2}}\iff
U_{1}(n_{1})\underset{k}{\cong}U_{2}(n_{2})\text{ for some }n_{1},n_{2}\in \mathbb{Z}_{\geq 0} 
\]
We call these two conjectures  the (weak) $m$-step solvable Grothendieck conjectures. In this survey, we explain the following three results of the $m$-step solvable Grothendieck conjecture.

\begin{thm}[\cite{Na1990-405} Theorem A, see Theorem \ref{thmnakamura}]\label{thm1}
Assume that  $m\geq 2$, and $k$ is an \textbf{NF} that satisfies one of the following (a)-(c).
\begin{enumerate}[(a)]
\item $k$ is the rational number field $\mathbb{Q}$.
\item $k$ is a quadratic field $\neq  \mathbb{Q}(\sqrt{2})$.
\item There exists a prime ideal $\mathfrak{p}$ of $O_{k}$ unramified in $k/\mathbb{Q}$ such that  $|O_k/\mathfrak{p}|=2$.
\end{enumerate}
Then  the  $m$-step solvable  Grothendieck conjecture   for  $4$-punctured  projective lines  over $k$  holds.
\end{thm}

\begin{thm}[\cite{Mo1999} Theorem A$^{'}$, see Theorem \ref{thmmochizuki}]\label{thm2}
Fix a prime $\ell$. Assume that $m\geq 5$, and $k$ is an \textbf{S}$\ell$\textbf{F}. Then the (pro-$\ell$)  $m$-step solvable  Grothendieck conjecture   for   hyperbolic curves over $k$  holds.
\end{thm}
\begin{thm}[\cite{Ya2020} Theorem 2.4.1, see Theorem \ref{thmnyamaguchi}]\label{thm3}
Assume that $m\geq 3$, and   $k$ is an $\textbf{FGF}$.  If, moreover, $p> 0$, we assume that  the curve  $X_{1,\overline{k}}-S'$ does not   descend to a curve over  $\overline{\mathbb{F}}_{p}$ for each $S'\subset E_{1,\overline{k}}$ with $|S'|=4$.  Then the  (pro-prime-to-$p$)  $m$-step solvable  Grothendieck conjecture   for  genus $0$ hyperbolic curves over $k$ under this assumption holds.
\end{thm}

In section $1$, we outline  the paper \cite{Na1990-405}, and we present the sketch of the  proof of Theorem \ref{thm1}. In section $2$,  we introduce the proof of Theorem \ref{thm2}. In section $3$, we outline the paper \cite{Ya2020}, and  we present the sketch of the  proof of Theorem \ref{thm3}.

\section{When   $m\geq 2$,  $(g,r)=(0,4)$  and $k$ is an \textbf{NF} with certain conditions}
 Through this section, we assume that $k$ is an \textbf{NF}, and  write   $O_{k}$ for  the ring of integers of $k$.   In this section, we introduce the paper  \cite{Na1990-405}. In  \cite{Na1990-405}, Nakamura exploited the theory of weights and  proved the following theorem. 

\begin{thm}[\cite{Na1990-405} Theorem A]\label{thmnakamura}
Let $\lambda_{i}\in k-\{0,1\}$ and  set $\Lambda_{i}=\{0,1,\infty,\lambda_{i}\}$, for each $i=1,2$. Assume that  $m\geq 2$, and $k$  satisfies one of the following (a)-(c).
\begin{enumerate}[(a)]
\item $k$ is the rational number field $\mathbb{Q}$.
\item $k$ is a quadratic field $\neq  \mathbb{Q}(\sqrt{2})$.
\item There exists a prime ideal $\mathfrak{p}$ of $O_{k}$ unramified in $k/\mathbb{Q}$ such that  $|O_k/\mathfrak{p}|=2$.
\end{enumerate}
Then the following holds.
\begin{equation*}
\Pi_{\mathbb{P}^{1}_{k}-\Lambda_{1}}^{(m)}\xrightarrow[G_{k}]{\sim} \Pi_{\mathbb{P}^{1}_{k}-\Lambda_{2}}^{(m)}\iff  \mathbb{P}^{1}_{k}-\Lambda_{1}\xrightarrow[k]{\sim}\mathbb{P}^{1}_{k}-\Lambda_{2}
\end{equation*}
In other words,   the  $m$-step solvable  Grothendieck conjecture   for  $4$-punctured  projective lines over $k$  holds.
\end{thm}

We will see the proof of Theorem \ref{thmnakamura} in this section. We divide the proof into the following  three steps. 
\begin{itemize}
\item[Step 1:]  Let   $\Lambda$  be a  finite subset of  $k\cup \{\infty\}$ which satisfies  $\{0,1,\infty\}\subset \Lambda$ and  $|\Lambda|\geq 4$. We  define $\Gamma(\Lambda)$ as the subgroup of $k^{\times}$ generated by $\{\lambda-\lambda'\mid\lambda,\lambda'\in \Lambda-\{\infty\}, \lambda\neq \lambda'\}$. We will show that,  for all $N\geq 1$, the field  $k(\Gamma(\Lambda)^{\frac{1}{N}})$ is reconstructed group-theoretically from $\Pi^{(2)}_{\mathbb{P}^{1}_{k}-\Lambda}$   (see Proposition \ref{rigidityprp}), where  $\Gamma(\Lambda)^{\frac{1}{N}}$ stands for the set of all elements  of $\overline{k}^{\times}$, whose $N$-th power is contained in $\Gamma(\Lambda)$. This step corresponds to \cite{Na1990-405} section 2.
\item[Step 2:]  For each $i=1,2$, let   $\Lambda_{i}$  be a  finite subset of  $k\cup \{\infty\}$ which satisfies  $\{0,1,\infty\}\subset \Lambda_{i}$ and  $|\Lambda_{i}|\geq 4$.  In this step, we will show that, if $k(\Gamma(\Lambda_{1})^{\frac{1}{\ell^{n}}})=k(\Gamma(\Lambda_{2})^{\frac{1}{\ell^{n}}})$ holds  for all primes $\ell$ and  all integers $n\geq 0$, then $\Gamma(\Lambda_{1})=\Gamma(\Lambda_{2})$ holds (see Proposition \ref{step2prop}).  This step corresponds to \cite{Na1990-405} section 3.
\item[Step 3:] For each $\lambda \in k-\{0,1\}$, we define the set:

\[
J(\lambda)=\left\{ \lambda, \frac{1}{\lambda},1-\lambda, \frac{1}{1-\lambda},\frac{\lambda}{\lambda-1},\frac{\lambda-1}{\lambda}\right\}.
\]
Let $\lambda_{i}\in k-\{0,1\}$. The curves $\mathbb{P}^{1}_{k}-\{0,1,\infty, \lambda_{1}\}$ and $\mathbb{P}^{1}_{k}-\{0,1,\infty, \lambda_{2}\}$ are isomorphic over $k$ if and only if $J(\lambda_{1})=J(\lambda_{2})$. In this step, we find out a certain relation between the equalities  $J(\lambda_{1})=J(\lambda_{2}) $ and $\Gamma(\{0,1,\infty,\lambda_{1}\})=\Gamma(\{0,1,\infty,\lambda_{2}\})$. In particular, we will prove the  main theorem of this section.  This step corresponds to \cite{Na1990-405} section 4.
\end{itemize}
\ \\
\underline{\textbf{Step 1}}\par
 Let  $\ell$ be a prime. First, we introduce the weight filtration of $\overline{\Pi}^{1,\text{pro-}\ell}_{U}$. Let  $ \mathbb{Z}[E(\overline{k})]$ be the free $\mathbb{Z}$-module generated by $E(\overline{k})$ and regard it as a  $G_k$-module via the  $G_k$-action on   $E(\overline{k})$.  We have the  following isomorphism and exact sequence of  $G_k$-modules.
\begin{equation}\label{wfseq}
\begin{cases}
\overline{\Pi}_{U}^{1,\text{pro-}\ell}\xrightarrow{\sim} T_{\ell}(J_{X})^{}&\ (r=0)\\
0\rightarrow \mathbb{Z}_{\ell}(1)\rightarrow \mathbb{Z}[E(\overline{k})]\bigotimes_{\mathbb{Z}} \mathbb{Z}_{\ell}(1)\xrightarrow{} \overline{\Pi}_{U}^{1,\text{pro-}\ell}\rightarrow T_{\ell}(J_{X})^{}\rightarrow 0&\  (r\neq 0).\\
\end{cases}
\end{equation}
Here $T_{\ell}(J_{X})$ stands for the $\ell$-adic Tate module of the Jacobian variety $J_{X}$ of $X$. The $G_k$-actions on $\mathbb{Z}[E(k^{\text{sep}})]\underset{\mathbb{Z}}\otimes\mathbb{Z}_{\ell}(1)$ and $T_{\ell}(J_{X})$ have weights $-2$ and $-1$ (see  \cite{Na1990-405} section 2).    The same assertion is true if $k$ is  an \textbf{FGF} and $\ell\neq  p$, see \cite{Ya2020} subsection 1.3.\par 
For an open subgroup $H\subset \Pi^{(m)}_{U}$ containing $\overline{\Pi}^{[m-1]}_{U}/\overline{\Pi}^{[m]}_{U}$, let $W_{-2}(\overline{H}^{1,\text{pro-}\ell})$ be  the unique  maximal $\text{pr}(H)$ submodule  of $\overline{H}^{1,\text{pro-}\ell}$  of weight $-2$  (see \cite{Na1990-405} (2.1)Proposition). We have the  Galois representation $\phi_{H}^{(\ell)}:\text{pr}(H)\rightarrow \text{Aut}(W_{-2}(\overline{H}^{1,\text{pro-}\ell}))$ for all primes $\ell$.  Using these Galois representations, we define the strong rigidity invariant. (In \cite{Na1990-405}, the strong rigidity invariant is defined  for  an arbitrary $G_{k}$-augmented profinite group (see \cite{Na1990-405} section 2), but we only consider the  strong rigidity invariant for $\Pi^{(2)}_{U}$.)

\begin{dfn}[cf. \cite{Na1990-405} (2.2)Definition]
\begin{enumerate}[(1)]
\item Let $\overline{H}$ be an open subgroup of $\overline{\Pi}^{m}_{U}$, and $G$ an open subgroup of $G_{k}$. We say that an open subgroup $H$ of $\Pi^{(m)}_{U}$ is  a Galois $G$-model of $\overline{H}$ if  $H\cap \overline{\Pi}^{m}_{U}=\overline{H}$, $\text{pr}(H)=G$, and $\text{pr}^{-1}(G)\rhd H$.
\item Let $N\in\mathbb{Z}_{\geq 1}$, and $\overline{\Pi}^{1,N\text{-th}}_{U}$ the set of the $N$-th powers of all elements of  $\overline{\Pi}_{U}^{1}$.   We define the $strong$ $rigidity$ $invariant$ $\kappa_{N}=\kappa_{N}(\Pi^{(2)}_{U})$ to be the subfield of $\overline{k}$ consisting of  the elements fixed by all the automorphisms of $\overline{k}$  belonging to 
\begin{equation*}
\bigcup_{\ell:\text{prime}}\bigcup_{H\in \mathscr{H}_{N}}\text{Ker}(\phi_{H}^{(\ell)})
\end{equation*}
where $\mathscr{H}_{N}$  is the set of all open subgroups $H$  of $\Pi^{(2)}_{U}$ containing $\overline{\Pi}^{[1]}_{U}/\overline{\Pi}^{[2]}_{U}$ such that  the image of $H$ in $\Pi^{(1)}_{U}$ is a Galois $G_{k(\mu_{N})}$-model of $\overline{\Pi}^{1,N\text{-th}}_{U}$
\end{enumerate}
\end{dfn}

\begin{prp}[cf. \cite{Na1990-405} (2.9)Theorem, \cite{Na1990-405} (2.10)Corollary]\label{rigidityprp}
 Let   $\Lambda$  be a  finite subset of  $k\cup \{\infty\}$ which satisfies  $\{0,1,\infty\}\subset \Lambda$ and  $|\Lambda|\geq 4$. For each $N\in\mathbb{Z}_{\geq 1}$, the following holds.
\begin{equation*}
\kappa_N(\Pi^{(2)}_{\mathbb{P}^{1}_{k}-\Lambda})\ =\ k(\Gamma(\Lambda)^{\frac{1}{N}})
\end{equation*}
\end{prp}
\begin{proof}[Sketch of Proof]
Set $\tilde{k}=k(\mu_{N})$.  We fix $H\in\mathscr{H}_{N}$. For a closed subgroup $W$ of  $\Pi^{(2)}_{\mathbb{P}^{1}_{k}-\Lambda}$, let $U_{W}$ be the cover of $\mathbb{P}_{k}^{1}-\Lambda$ corresponding to  $W\subset \Pi^{(2)}_{\mathbb{P}^{1}_{k}-\Lambda}$, and $E_{W}$ the inverse image of $E$ by the natural projection $U_{W}^{\text{cpt}}\twoheadrightarrow X$.  By \cite{Na1990-405} (2.7)Proposition, the fixed field of  $\text{Ker}(\phi_{H}^{(\ell)})$ coincides with $\mathbb{Q}(\mu_\ell)\cdot\tilde{k}(E_{H})$.  Thus, first, we consider $\tilde{k}(E_{H})$.  We have that $K(U_{\overline{\Pi}^{1,N\text{-th}}_{U}})$ coincides with $\overline{k}((t-\lambda)^{\frac{1}{N}} \mid \lambda\in\Lambda-\{\infty\})$ $(=$ the composite field of  $\{\overline{k}((t-\lambda)^{\frac{1}{N}})\}_{\lambda\in\Lambda-\{\infty\}})$. Then we obtain that $K(U_{H})$ coincides with the composite field of 
$\{\overline{k}((t-\lambda)^{\frac{1}{N}})\cap K(U_{H})\}_{\lambda\in\Lambda-\{\infty\}}$. 
Since $\overline{k}((t-\lambda)^{\frac{1}{N}})\cap K(U_{H})$ is a $\mathbb{Z}/N\mathbb{Z}$-cover of $\mathbb{P}^{1}_{\tilde{k}}-\{\lambda,\infty\}$, there exists  $\epsilon_{\lambda}\in \tilde{k}^{\times}$ such that:
\[
\overline{k}((t-\lambda)^{\frac{1}{N}})\cap K(U_{H})=\tilde{k}((\epsilon_{\lambda}(t-\lambda))^{\frac{1}{N}})
\]
by Kummer theory. Hence we obtain that  $\tilde{k}(E_{H})=\tilde{k}((\epsilon_{\lambda'}(\lambda-\lambda'))^{\frac{1}{N}}\mid \lambda,\lambda'\in\Lambda-\infty)$. Sine  $\tilde{k}(E_{H})\ni  (\epsilon_{1}(0-1))^{\frac{1}{N}}$, $(\epsilon_{0}(1-0))^{\frac{1}{N}}$, $(\frac{\epsilon_{\lambda}(0-\lambda)}{\epsilon_{0}(\lambda-0)})^{\frac{1}{N}}=(-\frac{\epsilon_{\lambda}}{\epsilon_{0}})^{\frac{1}{N}}$, $(\frac{\epsilon_{\lambda}(1-\lambda)}{\epsilon_{1}(\lambda-1)})^{\frac{1}{N}}=(-\frac{\epsilon_{\lambda}}{\epsilon_{1}})^{\frac{1}{N}}$ ($\lambda\in\Lambda-\{0,1,\infty\})$
, we get  $\tilde{k}(E_{H})\supset k(\Gamma(\Lambda)^{\frac{1}{N}})$. There exists an open subgroup  $H_{0}$ of $\Pi^{(2)}_{U}$  containing $\overline{\Pi}^{[1]}_{U}/\overline{\Pi}^{[2]}_{U}$ such that  the image of $H_{0}$ in $\Pi^{(1)}_{U}$ is a Galois $G_{k(\mu_{N})}$-model of $\overline{\Pi}^{1,N\text{-th}}_{U}$ and that $\epsilon_{\lambda}=1$  for all $\lambda\in\Lambda-\{\infty\}$.  We have    $\tilde{k}(E_{H_{0}})=k(\Gamma(\Lambda)^{\frac{1}{N}})$.  Thus, we obtain
\[
\kappa_N(\Pi^{(2)}_{\mathbb{P}^{1}_{k}-\Lambda})=\underset{\ell}\bigcap( \mathbb{Q}(\mu_\ell)\cdot k(\Gamma(\Lambda)^{\frac{1}{N}}))=k(\Gamma(\Lambda)^{\frac{1}{N}}).
\]
\end{proof}

\noindent\underline{\textbf{Step 2}}\par

\begin{prp}[cf. \cite{Na1990-405} (3.1)Lemma]\label{step2prop}
 Let  $k'$ be a field finitely generated over $\mathbb{Q}$, and $\Gamma_{1}$, $\Gamma_{2}$ finitely generated subgroups of  $k'^{\times}$. If  $k'(\Gamma_{1}^{\frac{1}{\ell^{n}}})=k'(\Gamma_{2}^{\frac{1}{\ell^{n}}})$ for all $n\in \mathbb{Z}_{\geq0}$ and  all primes $\ell$,  then $\Gamma_{1}=\Gamma_{2}$.
\end{prp}

\begin{proof}[Sketch of Proof]
We may replace $k'$ with a field finitely generated over $k'$, hence we may assume that $k'\supset \mu_{4}$.  Replacing $\Gamma_{2}$  with $\Gamma_{1}\Gamma_{2}$ if necessary, we may assume $\Gamma_{1}\subset \Gamma_{2}$. Let $\gamma\in \Gamma_{2}$. We have $k'(\mu_{\ell^{n}},\gamma^{\frac{1}{\ell^{n}}})\subset k'(\Gamma_{2}^{\frac{1}{\ell^{n}}})=k'(\Gamma_{1}^{\frac{1}{\ell^{n}}})$ and then $\gamma\in \Gamma_{1}\cdot k'(\mu_{\ell^{n}})^{\times \ell^{n}}$ by Kummer theory. By \cite{Ya2020} Lemma  2.1.3 (where we need the assumption $k'\supset \mu_{4}$), we obtain $\gamma\in (\Gamma_{1}\cdot k'(\mu_{\ell^{n}})^{\times \ell^{n}})\cap k'= \Gamma_{1}\cdot k'^{\times \ell^{n}}.$ Let $R$ be the integral closure of $\mathbb{Z}[\Gamma_{2}]$ in $k'$.  We have that $R^{\times}$ is a finitely generated  $\mathbb{Z}$-module because  $R$  is finitely generated over $\mathbb{Z}$.  We have $\gamma\in\Gamma_{1}\cdot R^{\times \ell^n}$  for all $n\in \mathbb{Z}_{\geq0}$ and  all primes $\ell$. Thus, we obtain $\gamma\in\Gamma_{1}$.
 \end{proof}

\ \\
\underline{\textbf{Step 3}}\par
In this step, we find out a relation between the equalities  $J(\lambda_{1})=J(\lambda_{2}) $ and $\Gamma(\{0,1,\infty,\lambda_{1}\})=\Gamma(\{0,1,\infty,\lambda_{2}\})$, and show Theorem \ref{thmnakamura}.  We write  $\mathcal{E}$ for the set of all elements $e\in \overline{\mathbb{Z}}^{\times}$ which satisfies $(1-e)\mid2$. 

\begin{prp}[cf. \cite{Na1990-405} (4.7)Theorem]\label{step3prop}
Let $\lambda_{1}$, $\lambda_{2}\in k-\{0,1\}$. Suppose that:
\begin{enumerate}[(i)]
\item $\Gamma(\{0,1,\infty,\lambda_{1}\})=\Gamma(\{0,1,\infty,\lambda_{2}\})$, 
\item $J(\lambda_{1})$ contains no elements of $\mathcal{E}$.
\end{enumerate}
Then $J(\lambda_{1})=J(\lambda_{2})$
\end{prp}
\begin{proof}[Sketch of Proof]
Since   $\Gamma(\{0,1,\infty,\lambda_{i}\})=\langle-1,\lambda_{i},1-\lambda_{i}\rangle$, we obtain 
\begin{itemize}
\item[(\text{a})]$\langle-1,\lambda_{1},1-\lambda_{1}\rangle=\langle-1,\lambda_{2},1-\lambda_{2}\rangle$
\end{itemize}
by (i).  
When $J(\lambda_{1})\cap O_{k}^{\times}\not=\emptyset$, we may assume that  $\lambda_{1}\in O_{k}^{\times}$ because  the conditions (i)(ii)  and the conclusion are preserved if  $\lambda_{i}$ is replaced with  another element of $J(\lambda_{i})$.  Then we get  $(1-\lambda_{1})\in O_{k}-O_{k}^{\times}$ by (ii). Hence   there exists an additive discrete valuation $v$ of $k$ such that $v(1-\lambda_{1})>0$.  When $J(\lambda_{1})\cap O_{k}^{\times}=\emptyset$, there exists an additive discrete valuation $v$ of $k$ such that $v(1-\lambda_{1})\not=0$.  Replacing $\lambda_{1}$ by $\frac{\lambda_{1}}{\lambda_{1}-1}$ if necessary, we obtain that $v(1-\lambda_{1})>0$.  Thus, the following holds by replacing  $\lambda_{1}$ with  another element of $J(\lambda_{i})$.
\begin{itemize}
\item[(\text{b})] There exists an additive discrete valuation $v$ of $k$ such that $v(1-\lambda_{1})>0$. 
\end{itemize}
If $(1-\lambda_{1})(1+\lambda_{1}^{n})=2$ for some $n\in \mathbb{Z}$, then $\lambda_{1}\in O_{k}^{\times}$ and we get  $\lambda_{1}\in \mathcal{E}$. Thus, the following holds by (ii)
\begin{itemize}
\item[(\text{c})] $(1-\lambda_{1})(1+\lambda_{1}^{n})\neq 2$ for all $n\in\mathbb{Z}$.
\end{itemize}
 By  \cite{Na1990-405} (4.8)Lemma and (a)(b)(c),  we get $J(\lambda_{1})=J(\lambda_{2})$.
\end{proof}

\begin{proof}[Sketch of Proof of Theorem \ref{thmnakamura}]
 Already we have $\Gamma(\{0,1,\infty,\lambda_{1}\})=\Gamma(\{0,1,\infty,\lambda_{2}\})$ by Proposition \ref{rigidityprp} and Proposition \ref{step2prop}.  If   $J(\lambda_{1})\cap \mathcal{E}=\emptyset$, then the assertion follows by  Proposition \ref{step3prop}. Thus, we may assume that $J(\lambda_{i})\cap \mathcal{E}\not =\emptyset$ for $i=1$, $2$. When $k=\mathbb{Q}$,  $J(\lambda_{1})=J(\lambda_{2})=\{-1,2,\frac{1}{2}\}$ by $\mathbb{Q}\cap \mathcal{E}=\{-1\}$.   Hence the assertion follows. \par
When $k$ is a quadratic field $\neq \mathbb{Q}(\sqrt{2})$, we can directly show that $e \in k\cap \mathcal{E}$ is either
\[
-1,\ \pm\sqrt{-1},\ \frac{1\pm\sqrt{-3}}{2},\ 2\pm\sqrt{3},\ \frac{3\pm\sqrt{5}}{2},\ \frac{\pm1\pm\sqrt{5}}{2},\ \text{ or }\ \pm2\pm\sqrt{5}.
\]
Hence we may only consider the cases where $k=\mathbb{Q}(\sqrt{-1}),$ $\mathbb{Q}(\sqrt{3})$, $\mathbb{Q}(\sqrt{-3})$ and $\mathbb{Q}(\sqrt{5})$.  In these cases, we can show that $\Gamma(\{0,1,\infty,\lambda_{1}\})\not=\Gamma(\{0,1,\infty,\lambda_{2}\})$  for all $\lambda_{1}$, $\lambda_{2}\in k\cap \mathcal{E}$ with $\lambda_{1}\neq \lambda_{2}$.  Thus, the assertion follows. \par
When  there exists a prime ideal $\mathfrak{p}$ of $O_{k}$ unramified in $k/\mathbb{Q}$ such that  $|O_k/\mathfrak{p}|=2$, we  show that  every element $\lambda_{1}\in k\cap \mathcal{E}$ satisfies the conditions (b) and (c) in the proof of Proposition \ref{step3prop}. Let $v$ denote an additive discrete valuation of $k$ corresponding to $\mathfrak{p}$.  By $|O_k/\mathfrak{p}|=2$, we have  $\lambda_{1}-1\in \mathfrak{p}$. Then $v(1+ \lambda_{1}^{n})>0$ for all $n\in \mathbb{Z}$.  In particular,  the condition (b)  in the proof of Proposition \ref{step2prop} follows. If $(1-\lambda_{1})(1+\lambda_{1}^{n})=2$ for some $n\in \mathbb{Z}$, then $2\in \mathfrak{p}^{2}$ because $1-\lambda_{1}$, $1+ \lambda_{1}^{n}\in \mathfrak{p}$. However, this  contradicts the assumption. Thus,   the condition (c) in the proof of Proposition \ref{step3prop}  follows. Now, the assertion follows by   \cite{Na1990-405} (4.8)Lemma. 
\end{proof}

\section{When $m\geq 5$ and  $k$ is an \textbf{S}$\ell$\textbf{F}}
 Through this section,  we fix a prime $\ell$ and  assume that $k$ is an $\textbf{S}\ell\textbf{F}$.  
In the paper \cite{Mo1999}, Mochizuki proved  the Grothendieck conjecture for hyperbolic curves over $k$ (\cite{Mo1999} Theorem A) in a much wider form than the one defined in the Introduction.  Moreover, Mochizuki  also showed  the $m$-step solvable Grothendieck conjecture for hyperbolic curves over $k$  (\cite{Mo1999} Theorem $A{'}$) as one of  the $m$-step solvable  versions of the main theorem.  In this section, we introduce the following theorem.

\begin{thm}[\cite{Mo1999} Theorem A${'}$]\label{thmmochizuki}
Assume that $m\geq 5$, and $U_{2}$ is a hyperbolic curve over $k$.  Let   $\Phi_{m}:\Pi_{U_{1}}^{(m,\text{\rm{pro}-}\ell)}\xrightarrow[G_{k}]{} \Pi^{(m,\text{\rm{pro}-}\ell)}_{U_{2}}$  be a continuous open $G_{k}$-homomorphism, and $\overline{\Phi}_{m-3}:\overline{\Pi}_{U_{1}}^{m-3,\text{\rm{pro}-}\ell}\rightarrow \overline{\Pi}^{m-3,\text{\rm{pro}-}\ell}_{U_{2}}$ the homomorphism defined by $\Phi_{m}$. Then $\Phi_{m}$  induces a unique dominant $k$-morphism $\mu: U_{1}\xrightarrow[k]{}U_{2}$ whose  induced homomorphism on geometric fundamental groups coincides (up to composition with an inner automorphism arising from $\overline{\Pi}^{m-3,\text{\rm{pro}-}\ell}_{U_{2}}$) with $\overline{\Phi}_{m-3}$. 
  In particular,    the (\text{pro}-$\ell$)  $m$-step solvable  Grothendieck conjecture   for   hyperbolic curves over $k$  holds.
\end{thm}

\begin{rem}
\begin{enumerate}
\item In \cite{Mo1999} Theorem A${'}$,   it is not assumed that $U_{1}$ is a curve over $k$, but  instead, it is assumed that  $U_{1}$ is a (smooth) variety over $k$. 
\item  The uniqueness  assertion for $\mu$ in Theorem \ref{thmmochizuki} is not explicitly stated  in  \cite{Mo1999} Theorem A${'}$. However,  the uniqueness follows from the fact  that  a dominant  $k$-morphism $X_{1}\rightarrow X_{2}$ is determined by the open homomorphism $\overline{\Pi}_{X_{1}}^{1,\text{\rm{pro}-}\ell}\rightarrow \overline{\Pi}^{1,\text{\rm{pro}-}\ell}_{X_{2}}$ induced by $X_{1}\rightarrow X_{2}$ when $g_{2}\geq 2$. See the first paragraph of the proof of  \cite{Mo1999} Theorem 14.1.
\item Mochizuki also proved the $m$-step solvable Grothendieck conjecture  when $U_{1}$ is a pro-variety (i.e.,  a $k$-scheme which can be written as a projective limit of smooth varieties over $k$ such that the transition morphisms are all birational)   and $U_{2}$ is a hyperbolic pro-curve  (i.e.,  a $k$-scheme which can be written as a projective limit of smooth hyperbolic curves over $k$ such that the transition morphisms are all birational) and $m\geq 3m_{1}+6\text{\rm{tr}}_{k}(K(U_{1}))+2$, where $m_{1}$ stands for  the minimal transcendence degree over $\mathbb{Q}_{\ell}$ of all finitely generated field extensions of  $\mathbb{Q}_{\ell}$ that contain $k$. See \cite{Mo1999} Theorem A${''}$.
\end{enumerate}
\end{rem}

\begin{proof}[Outline of Proof]
The second assertion follows from the first assertion (we remark that this proof needs the uniqueness of $\mu$).  We consider the first assertion.  In the next paragraph,   we assume that: (i) $k$ is a finite extension of $\mathbb{Q}_{\ell}$; (ii)  $U_{2}$ is a non-hyperelliptic proper hyperbolic curve; (iii) $m =3$.

By \cite{Mo1999} Lemma 0.4, we reconstruct group-theoretically  the $\ell$-adic cohomology group $H^{1}(U_{i,},\mathbb{Q}_{\ell})$ from $\Pi_{U_{i}}^{(1,\text{pro-}\ell)}$.  Using the Hodge-Tate decomposition for $\ell$-adic cohomology,   we obtain $H^{0}(U_{i},\omega_{U_{i}/k})$ as the $G_{k}$-invariant part of  $H^{1}(U_{i},\mathbb{Q}_{\ell})\otimes_{\mathbb{Q}_{\ell}}\mathbb{C}_{\ell}(1)$, where $\mathbb{C}_{\ell}(1)$ is the Tate-twist of $\mathbb{C}_{\ell}$. Hence, we get a (an injective) morphism $\theta:H^{0}(U_{2},\omega_{U_{2}/k})\rightarrow H^{0}(U_{1},\omega_{U_{1}/k})$ from $\Phi_{m}$, which induces a  morphism $\mathbb{P}(\theta): \mathbb{P}(H^{0}(U_{1},\omega_{U_{1}/k}))\rightarrow \mathbb{P}(H^{0}(U_{2},\omega_{U_{2}/k}))$. We have a canonical morphism $X_{i}\rightarrow \mathbb{P}(H^{0}(X_{i},\omega_{U_{i}/k}))$ for $i=1,2$. Moreover,  since $U_{2}$ is non-hyperelliptic proper hyperbolic,  $X_{2}=U_{2}$ and $U_{2}\rightarrow \mathbb{P}(H^{0}(U_{2},\omega_{U_{2}/k}))$ is an embedding.    For each  positive integer  $i$, we set:
\[
\mathcal{R}^{i}:=\text{Ker}(\bigotimes^{i} H^{0}(U_{2},\omega_{U_{2}/k})\rightarrow H^{0}(U_{2},\omega^{\otimes i}_{U_{2}/k})), 
\]
which is the  set of relations of degree $i$ defining $U_{2}\hookrightarrow \mathbb{P}(H^{0}(U_{2},\omega_{U_{2}/k}))$.  If the map:
\[
\kappa^{i}:\bigotimes^{i} H^{0}(U_{2},\omega_{U_{2}/k}) \xrightarrow{\otimes^{i}\theta}  \bigotimes^{i} H^{0}(U_{1},\omega_{U_{1}/k})\rightarrow H^{0}(U_{1},\omega^{\otimes i}_{U_{1}/k})
\]
satisfies $\kappa^{i}(\mathcal{R}^{i})=0$ for all $i$,  we say that  $\theta$ preserves relations.  When   $\theta$ preserves relations,  $\mathbb{P}(\theta)$ induces a morphism $U_{1}\rightarrow U_{2}$. Thus, we have only  to show that the morphism $\theta$ (induced by  $\Phi_{3}:\Pi_{U_{1}}^{(3,\text{\rm{pro}-}\ell)}\rightarrow \Pi^{(3,\text{\rm{pro}-}\ell)}_{U_{2}}$) preserves relations.  Let $\mathcal{X}_{1}$ be a (proper) model of $U_{1}$ over $O_{k}$, and $\mathfrak{p}$ an irreducible component of a special fiber of $\mathcal{X}_{1}$. We write $L$ for the quotient field of the completion of the local ring $O_{\mathcal{X}_{1},\mathfrak{p}}$. From the definition of $L$, we have a natural $L$-valued point $\text{Spec}(L)\rightarrow U_{1}$. In particular, we get  $\alpha^{L}:G_{L}\rightarrow \Pi_{U_{1}}^{(1,\text{pro-}\ell)}\rightarrow \Pi_{U_{2}}^{(1,\text{pro-}\ell)}$. If the morphism $\alpha^{L}$ is geometric (i.e.,  comes from some $L$-valued point $ \text{Spec}(L)\rightarrow U_{2}$), then $\alpha^{L}$ preserves relations. Thus, we have only to  show that $\alpha^{L}$ is geometric. This step is the main part of the first half of \cite{Mo1999}.  To construct a line bundle of degree prime to $\ell$ in this step, we need $\Pi_{U_{2}}^{(3,\text{pro-}\ell)}$. For more detail,  see \cite{Mo1999} sections 1-7, 18. \par
For the assumption (i), see  \cite{Mo1999} section $15$.  For arbitrary  hyperbolic curve $U_{2}$,  there always exists a covering $V_{2}\rightarrow U_{2}$ corresponding to  an open subgroup of  $\Pi_{U_{2}}^{(2,\text{pro-}\ell)}$ such that $g(V_{2})\geq 2$, that  the compactification   $V_{2}^{\text{cpt}}$ is non-hyperelliptic, and  that $V_{2}^{\text{cpt}}\rightarrow U_{2}^{\text{cpt}}$ has arbitrarily large (specified) ramification over all points of $U_{2}^{\text{cpt}}\rightarrow U_{2}$. By using this cover, we can lift the assertion (ii) if $m\geq 5$. 
\end{proof}

\section{When $m\geq 3$,   $g=0$ and  $k$ is an  $\textbf{FGF}_{\infty}$.}

 Through this section, we assume that $k$ is an $\textbf{FGF}$.  In the paper \cite{Na1990-411}, Nakamura proved the  Grothendieck conjecture   for genus $0$ hyperbolic curves over fields finitely generated over $\mathbb{Q}$.   A key step of the proof in  \cite{Na1990-411} is to reconstruct group-theoretically the inertia groups of $\overline{\Pi}_{U}$ from $\Pi_{U}$.  In the paper \cite{Ya2020}, the author proved the following   theorem  by giving the  $m$-step solvable version of   the method of the proof  in  \cite{Na1990-411}.

\begin{thm}[\cite{Ya2020} Theorem 2.4.1]\label{thmnyamaguchi}
Assume that $m\geq 3$, and  $U_{1}$ is a genus $0$ hyperbolic curve over $k$.  If, moreover, $p> 0$, we assume that:
$(\dag)$ For each $S'\subset E_{1, \overline{k}}\ \text{with}\ |S'|=4$, the curve  $X_{1,\overline{k}}-S'$ does not  descend to a curve over  $\overline{\mathbb{F}}_{p}$. Then the following holds.
\[
\Pi^{(m,p')}_{U_{1}}\underset{G_k}{\cong} \Pi^{(m,p')}_{U_{2}}\iff
\begin{cases}
U_{1}\underset{k}{\cong}U_{2} & (p=0) \\
U_{1}(n_{1})\underset{k}{\cong}U_{2}(n_{2})\text{ for some }n_{1},n_{2}\in \mathbb{Z}_{\geq 0} & (p>0)
\end{cases}
\]
  In other words,   the (pro-prime to $p$) $m$-step solvable  Grothendieck conjecture   for  genus $0$ hyperbolic curves over $k$ under $(\dag$)  holds.
\end{thm}

In this section, we introduce the paper  \cite{Ya2020} and a sketch of the proof of Theorem \ref{thmnyamaguchi}.  In subsection 3.1,  we explain the group-theoretical reconstruction of inertia groups.  In subsection 3.2,  we  explain the outline of the proof  of  Theorem \ref{thmnyamaguchi}  in detail. In subsection 3.3, we discuss extensions of Theorem \ref{thmnyamaguchi}.

\subsection{The group-theoretical reconstruction of inertia groups}
For a scheme $S$, denote by $S^{\text{cl}}$ the set of all closed points of  $S$. We define $\tilde{X}^{m,\Sigma}$ as the  cover of $X$ corresponding to $\{1\}\subset \Pi^{(m,\Sigma)}_{U}$, and write  $\tilde{U}^{m,\Sigma}$ and  $\tilde{E}^{m,\Sigma}$ for the inverse image  of $U$ and $E$ by $\tilde{X}^{m,\Sigma}\twoheadrightarrow X$, respectively. Let $I_{\tilde{v},\overline{\Pi}^{m,\Sigma^{\dag}}_{U}}\ (\text{resp. }D_{\tilde{v},\Pi^{(m,\Sigma^{\dag})}_{U}})$ be the stabilizer of $\tilde{v}\in (\tilde{X}^{m,\Sigma^{\dag}})^{\text{cl}}$ in $\overline{\Pi}^{m,\Sigma^{\dag}}_{U}$ (resp. $\Pi^{(m,\Sigma^{\dag})}_{U}$) with respect to the natural action $\overline{\Pi}^{m,\Sigma^{\dag}}_{U}\curvearrowright(\tilde{X}^{m,\Sigma^{\dag}})^{\text{cl}}\ (\text{resp. }\Pi^{(m,\Sigma^{\dag})}_{U}\curvearrowright(\tilde{X}^{m,\Sigma^{\dag}})^{\text{cl}}$). We call it the  inertia group (resp. the decomposition group) at $\tilde{v}$. For each $x\in E$, we set 
\[
I_{x,\overline{\Pi}^{m,\Sigma^{\dag}}_{U}}:=\overline{\langle I_{\tilde{v},\overline{\Pi}^{m,\Sigma^{\dag}}_{U}}\mid \tilde{v}\in (\tilde{X}^{m,\Sigma^{\dag}})^{\text{cl}}, \tilde{v}\text{ is above }x\rangle},
\]
and set $I_{\overline{\Pi}^{m,\Sigma^{\dag}}_{U}}:=\overline{\langle I_{\tilde{v},\overline{\Pi}^{m,\Sigma^{\dag}}_{U}}\mid \tilde{v}\in (\tilde{X}^{m,\Sigma^{\dag}})^{\text{cl}}\rangle}$.  In this subsection, we show the following proposition.

\begin{prp}[cf. \cite{Na1990-411} (3.5)Corollary, \cite{Ya2020} Corollary 1.4.8] \label{1.4.5}
Assume either ``$m\geq 2$ and  $r\geq2$" or ``$m=1$ and $r\geq 3$".  Let $\Phi_{m+2}^{\dag}:\Pi_{U_{1}}^{(m+2,\Sigma^{\dag})}\xrightarrow[G_{k}]{\sim}\Pi_{U_{2}}^{(m+2,\Sigma^{\dag})}$ be a  $G_{k}$-isomorphism, and  $\Phi_{m}^{\dag}:\Pi_{U_{1}}^{(m,\Sigma^{\dag})}\xrightarrow[G_{k}]{\sim}\Pi_{U_{2}}^{(m,\Sigma^{\dag})}$ the isomorphism induced  by $\Phi_{m+2}^{\dag}$. Then $\Phi_{m}^{\dag}$ preserves the decomposition groups at cusps. 
\end{prp}

We divide the proof  of Proposition \ref{1.4.5} into the following  two steps. 

\begin{itemize}
\item[Step 1:] 
Let  $\mathcal{F}$ be a free pro-$\Sigma$ group,   $X \subset  \mathcal{F}$  a set of free generators of $\mathcal{F}$, and $x\in X$.  For each element $f\in\mathcal{F}$, we wirte $Z_{\mathcal{F}^{m}}(f)$  for  the centralizer of $f$ in $\mathcal{F}^{m}$. First, we show that   $Z_{\mathcal{F}^{m}}(x^{n}) = \overline{\langle x\rangle}$ if $m \geq  2$ and $n\in\mathbb{Z}-\{0\}$ (see Lemma \ref{inertiaprp1}).  By using this result,   we get $D_{y,\Pi_{U}^{(m,\Sigma^{\dag})}}=N_{\Pi_{U}^{(m,\Sigma^{\dag})}}(I_{y,\overline{\Pi}_{U}^{m,\Sigma^{\dag}}})$  for all $ y\in \tilde{E}^{m,\Sigma^{\dag}}_{U}$.

\item[Step 2:]  We consider the group-theoretical reconstruction of inertia groups of $\overline{\Pi}_{U}^{m,\Sigma^{\dag}}$ from $\Pi_{U}^{(m+2,\Sigma^{\dag})}$.  For this,  we use the maximal cyclic subgroups of cyclotomic type (see Definition \ref{defcmct}). Since  $D_{y,\Pi_{U}^{(m,\Sigma^{\dag})}}=N_{\Pi_{U}^{(m,\Sigma^{\dag})}}(I_{y,\overline{\Pi}_{U}^{m,\Sigma^{\dag}}})$ ($ y\in \tilde{E}^{m,\Sigma^{\dag}}_{U}$), we obtain the  group-theoretical reconstruction of decomposition  groups at cusps of $\Pi_{U}^{(m,\Sigma^{\dag})}$ from $\Pi_{U}^{(m+2,\Sigma^{\dag})}$.
\end{itemize}

\noindent\underline{\textbf{Step 1}}\par
Let  $\mathcal{F}$ be a free pro-$\Sigma$ group and   $X \subset  \mathcal{F}$  a set of free generators of $\mathcal{F}$. We write   $Z_{\mathcal{F}^{m}}(y) $  for  the centralizer of $y$ in $\mathcal{F}^{m}$  ($y\in \mathcal{F}$). 

\begin{lmm}[cf. \cite{Ya2020} Proposition 1.1.10]\label{inertiaprp1}
Let $n \in \mathbb{Z}-\{0\}$ and $x\in X$.  If $m \geq  2$,   then  $Z_{\mathcal{F}^{m}}(x^{n}) = \overline{\langle x\rangle}$. In particular,   $\mathcal{F}^{m}$ is either  abelian or center-free.
\end{lmm}
\begin{proof}[Sketch of Proof]
Let $\alpha\in \hat{\mathbb{Z}}^{\Sigma}-\left\{0\right\}$.  By \cite{Ya2020} Proposition 1.1.6,  we obtain  that $Z_{\mathcal{F}^m}(x^\alpha) \ \subset\  \overline{\langle x\rangle}\cdot\mathcal{F}^{[m-1]}/\mathcal{F}^{[m]}$. (The proof of  \cite{Ya2020} Proposition 1.1.6 is essentially the  same as \cite{Na1994} Lemma 2.1.2.)  For simplicity, we only consider the case of  $|X|<\infty$.  
A  calculation shows that     $\hat{\mathbb{Z}}^{\Sigma}[[\mathcal{F}^{1}]]\ni x^n-1$ is a non-zero-divisor   (\cite{Ya2020} Lemma 1.1.7(1)). 
In \cite{Ya2020} Appendix,  we establish the  Blanchfield-Lyndon theory for $\mathcal{F}$. By using this theory,  we obtain that $\hat{\mathbb{Z}}^{\Sigma}[[\mathcal{F}^{1}]]\ni x^\alpha-1$ is a non-zero-divisor if and only if $Z_{\mathcal{F}^2}(x^\alpha)=\overline{\langle x\rangle}$.  Hence we get  $Z_{\mathcal{F}^2}(x^\alpha)=\overline{\langle x\rangle}$. Finally,  the assertion follows by the induction on $m$.
\end{proof}
Let  $ y, y'\in \tilde{E}^{m,\Sigma^{\dag}}$ with $y\neq y'$, and assume that  $r\geq 2$ and $m\geq 2$.   Since $\overline{\Pi}^{\Sigma^{\dag}}_{U}$ is a free pro-$\Sigma^{\dag}$ group and each  inertia group of $\overline{\Pi}^{\Sigma^{\dag}}_{U}$ is  generated by  a free generator of  $\overline{\Pi}^{\Sigma^{\dag}}_{U}$, we obtain  that  $I_{y,\overline{\Pi}^m}$ and $I_{y',\overline{\Pi}^m}$ are not commensurable   by Lemma \ref{inertiaprp1} if $y|_{X_{\overline{k}}}= y'|_{X_{\overline{k}}}$. When $y|_{X_{\overline{k}}}\neq  y'|_{X_{\overline{k}}}$, we also have that  $I_{y,\overline{\Pi}^m}$ and $I_{y',\overline{\Pi}^m}$ are not commensurable   by  \cite{Ya2020} Lemma 1.2.1. In particular, $I_{y,\overline{\Pi}_{U}^{m,\Sigma^{\dag}}}=N_{\overline{\Pi}_{U}^{m,\Sigma^{\dag}}}(I_{y,\overline{\Pi}_{U}^{m,\Sigma^{\dag}}})$ and   $D_{y,\Pi_{U}^{(m,\Sigma^{\dag})}}=N_{\Pi_{U}^{(m,\Sigma^{\dag})}}(I_{y,
\overline{\Pi}_{U}^{m,\Sigma^{\dag}}})$  hold (\cite{Ya2020} Proposition 1.2.3).
\ \\

\noindent\underline{\textbf{Step 2}}\par
We consider the group-theoretical reconstruction of inertia groups of $\overline{\Pi}_{U}^{m,\Sigma^{\dag}}$ from $\Pi_{U}^{(m+2,\Sigma^{\dag})}$.  For this,  we use the maximal cyclic subgroups of cyclotomic type.  They are   first defined in  \cite{Na1990-411} (3.3)Definition in the case of the full fundamental group. The following definition differs from that of \cite{Na1990-411} for the following two points; we weaken the self-normalizing property in \cite{Na1990-411}, and we generalize the definition from \textbf{NF}s to \textbf{FGF}s.

\begin{dfn}[cf. \cite{Na1990-411} (3.3)Definition, \cite{Ya2020} Definition 1.4.3]\label{defcmct}
Let $J$ be  a closed subgroup of $\overline{\Pi}^{m,\Sigma^{\dag}}_{U}$. If $J$ satisfies the following conditions, then $J$ is called a  maximal cyclic subgroup of cyclotomic type.
\begin{enumerate}[(i)]
\item $J \cong \hat{\mathbb{Z}}^{\Sigma^{\dag}}$
\item Write  $\overline{J}$ for the image   of  $J$ by $\overline{\Pi}_{U}^{m,\Sigma^{\dag}}\rightarrow \overline{\Pi}_{U}^{1,\Sigma^{\dag}}$. Then $J\overset{\sim}\rightarrow \overline{J}$ and  $\overline{\Pi}_{U}^{1,\Sigma^{\dag}}/\overline{J}$ is torsion-free.
\item $\text{pr}(N_{\Pi_{U}^{(m,\Sigma^{\dag})}}(J))\overset{\text{op}}\subset G_k$
\item Let $\chi_{\text{cycl}}:G_k\rightarrow (\hat{\mathbb{Z}}^{\Sigma^{\dag}})^{\times}$ be the cyclotomic character and $N_{\Pi_{U}^{(m,\Sigma^{\dag})}}(J)\to \text{Aut}(J)=(\hat{\mathbb{Z}}^{\Sigma^{\dag}})^{\times}$ the character obtained from the conjugate action. Then the  following diagram is commutative.
\[
\xymatrix@R=10pt{
N_{\Pi_{U}^{(m,\Sigma^{\dag})}}(J) \ar[r] \ar[d]& \text{Aut}(J)\ar@{}[d]|*=0[@]{\rotatebox{90}{=}}\\
G_k\ar[r]^{\chi_{\text{cycl}}}&(\hat{\mathbb{Z}}^{\Sigma^{\dag}})^{\times}\\
}
\]

\end{enumerate}
\end{dfn}
\begin{lmm}[cf. \cite{Ya2020} Lemma 1.4.1] \label{newlmm}
Let $Q$ be a finite group,   $G$ a profinite group, $J$ a closed subgroup of $G$, and  $\mathcal{I}$   a closed subset of  $G$.  If  there exists $\rho_{J}:J \twoheadrightarrow Q$ such that $\mathcal{I}\cap J\subset \text{ker}(\rho_{J})$, then there exist an open subgroup  $H\subset G$ with $J\subset H$ and a surjection $\rho_{H}:H \twoheadrightarrow Q$ such that $\mathcal{I}\cap H\subset \text{ker}(\rho_{H})$ and that $J\hookrightarrow H\overset{\rho_{H}}\twoheadrightarrow Q$ coincides with $\rho_{J}$. 
\end{lmm}
\begin{proof}[Proof]  By \cite{Za} Lemma 1.1.16(b), there exists an open subgroup $H_{0}\subset G$ with $J\subset H_{0}$ and a  morphism $\rho_{H_{0}}:H_{0}\rightarrow Q$ such that   the morphism $J\hookrightarrow H_{0}\overset{\rho_{H_{0}}} \rightarrow Q$ coincides with $\rho_{J}$. Let $\mathcal{H}:=\{ H\overset{\text{op}}\subset G\mid J\subset H\subset H_{0}\}$, and $\rho_{\tilde{H}}:\tilde{H}\hookrightarrow H_{0}\overset{\rho_{H_{0}}} \rightarrow Q$ for $\tilde{H}\in\mathcal{H}$. Since $(\mathcal{I}\cap \tilde{H})-\text{Ker}(\rho_{\tilde{H}})$ is closed in $G$ and $\underset{\tilde{H}\in\mathcal{H}}\cap ((\mathcal{I}\cap \tilde{H})-\text{Ker}(\rho_{\tilde{H}}))=(\mathcal{I}\cap F)-\text{Ker}(\rho_{F})=\emptyset$, there exists   $H\in\mathcal{H}$ such that  $\mathcal{I}\cap H\subset \text{ker}(\rho_{H})$ by the compactness argument.
\end{proof}
\begin{lmm}[cf. \cite{Na1990-411} (3.4)Theorem, \cite{Ya2020} Proposition 1.4.5] \label{1.4.2}
Assume that   $r\geq 2$. Then,  for any subgroup $I$ of  $\overline{\Pi}^{m,\Sigma^{\dag}}_{U}$, the following conditions are equivalent.
\begin{enumerate}[(a)]
\item $I$ is the inertia group at a point in $\tilde{E}^{m,\Sigma^{\dag}}$.
\item There exists a maximal cyclic subgroup of cyclotomic type $J$ of $\overline{\Pi}^{m+2,\Sigma^{\dag}}_{U}$ whose image by $\overline{\Pi}^{m+2,\Sigma^{\dag}}_{U}\rightarrow \overline{\Pi}^{m,\Sigma^{\dag}}_{U}$  coincides with $I$.
\end{enumerate}
\end{lmm}
\begin{proof}[Sketch of Proof] For simplicity, we set $\Delta:=(\overline{\Pi}_{U}^{\Sigma^{\dag}})^{[m+1]}/(\overline{\Pi}_{U}^{\Sigma^{\dag}})^{[m+2]}$. We  set    $\mathcal{I}:=\underset{\tilde{y}\in\tilde{X}^{m+2,\Sigma^{\dag}}}\cup I_{\tilde{y}, \overline{\Pi}_{U}^{m+2,\Sigma^{\dag}}}.$ For all  open subgroups $H$ of $(\overline{\Pi}_{U}^{\Sigma^{\dag}})^{m+2}$ containing $J\cdot \Delta$,  we obtain that $J\subset H^{[1]}\cdot  \langle \mathcal{I}\cap H\rangle$  by Definition \ref{defcmct}(iii)(iv) and $I_{H^{1}}=W_{-2}(H^{1})(:=\prod_{\ell\in \Sigma^{\dag}}W_{-2}(H^{1,\text{pro-}\ell}))$.  \par
We fix a prime $\ell$ with $J^{\ell} \subsetneq J$. Suppose that  $\mathcal{I}\cap (J\cdot \Delta)\subset J^{\ell}\cdot \Delta$.  Then there exists an  open subgroup $H$ of $(\overline{\Pi}_{U}^{\Sigma^{\dag}})^{m+2}$ containing $J\cdot \Delta$ and  $\rho_{H}: H \twoheadrightarrow \mathbb{Z}/\ell\mathbb{Z}$ such that $\rho_{H}(\mathcal{I} \cap H)=\{1\}$ and $\rho_{H}(J)\neq \{1\}$ by Lemma \ref{newlmm}.  This is absurd. Thus, we get $\mathcal{I}\cap (J\cdot \Delta)\not\subset  J^{\ell}\cdot \Delta$, and hence $\mathcal{I}\cap (J\cdot \Delta)\neq \{1\}$. Therefore, we obtain that $J\subset \mathcal{I}\cdot ((\overline{\Pi}_{U}^{\Sigma^{\dag}})^{[m]}/(\overline{\Pi}_{U}^{\Sigma^{\dag}})^{[m+2]}$ by \cite{Ya2020} Lemma 1.4.2.\par
Let $z\in J$ be a generator, and $\overline{z}$ the image of $z$ by $\overline{\Pi}_{U}^{m+2,\Sigma^{\dag}}\twoheadrightarrow \overline{\Pi}_{U}^{m,\Sigma^{\dag}}$. Since  there exists an inertia group $\tilde{I}$ that contains $\overline{z}$,  we have $\langle \overline{z}\rangle\equiv \tilde{I}$ mod  $(\overline{\Pi}^{\Sigma^{\dag}}_{U})^{[1]}/(\overline{\Pi}_{U}^{\Sigma^{\dag}})^{[m]}$ by Definition \ref{defcmct}(i)(ii). Since $\tilde{I}$ mapped injectively  by $\overline{\Pi}^{m,\Sigma^{\dag}}_{U}\twoheadrightarrow \overline{\Pi}^{1,\Sigma^{\dag}}_{U}$, we obtain $\langle \overline{z}\rangle=\tilde{I}$. \par
\end{proof}

\begin{proof}[Sketch of Proof of Proposition \ref{1.4.5}]
When $m\geq 2$ and  $r\geq2$, the assertion follows from Lemma \ref {1.4.2} and $D_{\tilde{y},\Pi_{U}^{(m,\Sigma^{\dag})}}=N_{\Pi_{U}^{(m,\Sigma^{\dag})}}(I_{\tilde{y},\overline{\Pi}_{U}^{m,\Sigma^{\dag}}})$ for $\tilde{y}\in \tilde{E}^{m,\Sigma^{\dag}}_{U}$. When $m=1$, the above method cannot  be used because $D_{y,\Pi_{U}^{(m,\Sigma^{\dag})}}\subsetneq N_{\Pi_{U}^{(m,\Sigma^{\dag})}}(I_{y,\overline{\Pi}_{U}^{m,\Sigma^{\dag}}})$ in general. In this case, we need to think about the maximal nilpotent quotient of $\Pi_{U}^{(\Sigma^{\dag})}$. For this, see the proof of \cite{Ya2020} Proposition 1.4.6.
 \end{proof}

\subsection{The proof of the main theorem}

 The important difference between the proofs of Theorem \ref{thmnakamura} and Theorem \ref{thmnyamaguchi} is Proposition \ref{1.4.5}. In Theorem \ref{thmnakamura}, we do not reconstruct  group-theoretically inertia groups. Thus, we could only consider the case of $\mathbb{P}_{k}^{1}$ minus $4$ points, but we get the results for the case of  $m=2$. In Theorem \ref{thmnyamaguchi}, we have the group-theoretical  reconstruction of  inertia groups. Thus, we can consider the case of all genus $0$ curves, but we should assume that $m\geq 3$. \par
In this subsection, we show Theorem \ref{thmnyamaguchi}.   As in section $1$, we divide the proof into the following  three steps. 
\begin{itemize}
\item[Step 1:]  Let   $\Lambda$  be a  finite subset of  $k\cup \{\infty\}$ with  $|\Lambda|\geq 4$. Let $x_1,x_2,x_3,x_4$ be distinct elements of $\Lambda$ and $\varepsilon=\left\{x_1,x_2\right\},\delta=\left\{x_3,x_4\right\}$. Set the following notation for $\varepsilon$ and $\delta$.
\[
\lambda  (\varepsilon, \delta):=\frac{ x_4-x_1 }{x_4-x_2} \frac{ x_3-x_2 }{x_3-x_1}
\]
(The isomorphism $\mathbb{P}^1_k\xrightarrow[k]{\sim} \mathbb{P}^1_k$ satisfying $x_1\mapsto 0,\ x_2\mapsto \infty,\ x_3\mapsto 1$ maps $x_4$  to $\lambda  (x_1,x_2,x_3,x_4)$.)    We will show that,  for all $n\in \mathbb{Z}_{\geq0}$ and  for all primes $\ell$ that differ from $p$, the field  $k(\mu_{\ell^{n}},\lambda(\varepsilon,\delta)^{\frac{1}{\ell^{n}}})$ is reconstructed group-theoretically from $\Pi^{(3)}_{\mathbb{P}_{k}^{1}-\Lambda}$   (see Proposition \ref{2.1.3}). This step corresponds to \cite{Ya2020} subsection 2.1.
\item[Step 2:]  
 Assume that $p=0$ (resp. $p>0$). For each  $i=1,2$, let   $\Gamma_{i}$ be a finitely generated subgroup of  $k^{\times}$.  In this step, we will show that,  if  $k(\Gamma_{1}^{\frac{1}{\ell^{n}}})=k(\Gamma_{2}^{\frac{1}{\ell^{n}}})$ for all $n\in \mathbb{Z}_{\geq0}$ and  for all primes $\ell$ that differ from $p$, then $\Gamma_{1}=\Gamma_{2}$ (resp.   $\underset{n\in\mathbb{Z}_{\geq0}}\cup \Gamma_{1}^{\frac{1}{p^{n}}}= \underset{n\in\mathbb{Z}_{\geq0}}\cup \Gamma_{2}^{\frac{1}{p^{n}}}$) holds (see Lemma \ref{nakamuralmm2}).   In particular, we get the  reconstruction of  $\lambda\in k^{\times}-\left\{1\right\}$  (resp. a non-torsion element $\lambda$ of $k^{\times}$) from  $\langle\lambda\rangle$ and $\langle1-\lambda\rangle$ in $k^{\times}$. This step corresponds to \cite{Ya2020} subsection 2.1.
\item[Step 3:] In this step, we first show Theorem  \ref{thmnyamaguchi} for punctured projective lines.  If $p=0$, then the result easily follows  from Step $1$ and  Step $2$. However, if $p>0$, then we have  to consider the gluing of the Frobenius twists.  Finally, we show Theorem \ref{thmnyamaguchi} in general by genus $0$ descent.  This step corresponds to \cite{Ya2020} subsections 2.2, 2.3 and 2.4. 
\end{itemize}

\noindent\underline{\textbf{Step 1}}\par
First, we define the weak rigidity invariant of  $\mathbb{P}^1_k-\Lambda$ (where $\Lambda$ stands for a finite set of $k$-rational points of $\mathbb{P}_k^1$ with $|\Lambda|\geq 4$). The weak rigidity invariant  was first defined in \cite{Na1990-411} (4.2) for the full fundamental group $\Pi_{\mathbb{P}^1_k-\Lambda}$ when $p=0$.  However, this definition is applicable to arbitrary $p\geq 0$ and  essentially uses only $\Pi_{\mathbb{P}^1_k-\Lambda}^{(1,p')}$. Thus,   in \cite{Ya2020} subsection 1.4,  the author redefined the weak rigidity invariant as follows.   
\begin{dfn}[cf. \cite{Na1990-411} (4.2), \cite{Ya2020} Definition 2.1.1]\label{2.1.1}
Let $n\in\mathbb{Z}_{\geq 0}$, $\ell$  a prime different from $p$,  and  $\Lambda$  a finite set of $k$-rational points of $\mathbb{P}_k^1$ with $|\Lambda|\geq 4$.  Let $x_1,x_2,x_3,x_4$ be distinct elements of $\Lambda$ and $\varepsilon=\left\{x_1,x_2\right\},\delta=\left\{x_3,x_4\right\}$.  We define  $\kappa_{\ell^{n}}(\varepsilon,\delta)$ to be the subfield of $k^{\text{sep}}$ consisting of  the elements fixed by all the automorphisms of $k^{\text{sep}}$ belonging to 
\begin{equation*}
\bigcup_{H}\bigcap_{\tilde{y}}(H\cap D_{\tilde{y}}).
\end{equation*}
Here $\tilde{y}$ run over all closed points of $\tilde{\Lambda}^{1,p'}$ above $\delta$, and  $H$ run over all open subgroups of $\Pi^{(1,p')}_{\mathbb{P}^1_k-\Lambda}$ satisfying the following conditions.
\begin{enumerate}[(i)]
\item  $\overline{H}:=H\cap \overline{\Pi}_{\mathbb{P}^1_k-\Lambda}^{1,p'}$ contains $I_{x, \overline{\Pi}_{\mathbb{P}^1_k-\Lambda}^{1,p'}}$ for all $x\in E-\varepsilon$, 
\item $\overline{\Pi}_{\mathbb{P}^1_k-\Lambda}^{1,p'}/\overline{H} \cong \mathbb{Z}/\ell^{n}\mathbb{Z}$
\item $\text{pr}(H)=G_{k(\mu_{\ell^{n}})}$
\item $\text{pr}^{-1}(G_{k(\mu_{\ell^{n}})})\rhd H$
\end{enumerate}
We call  $\kappa_{\ell^n}(\varepsilon,\delta)$ the $weak$ $rigidity$ $invariant$ for $\varepsilon, \delta$ of $U$.  
\end{dfn}
In \cite{Ya2020}, the weak rigidity invariant $\kappa_{N}(\varepsilon,\delta)$ is defined for all integers $N\in\mathbb{Z}_{\geq 1}$, whose all prime factors are contained in  $\Sigma^{\dag}$.  In the  following  proofs, we  only need  the weak rigidity invariant   in the case of $N=\ell^{n}$.
\begin{prp}[cf. \cite{Na1990-411}(4.3)Proposition, \cite{Ya2020} Proposition 2.1.2]\label{2.1.3} Under the notation of Definition \ref{2.1.1}, the following holds.
\begin{equation*}
\kappa_{\ell^n}(\varepsilon,\delta)\ =\ k(\mu_{\ell^n},\lambda(\varepsilon,\delta)^{\frac{1}{\ell^{n}}})
\end{equation*}
\end{prp}

\begin{proof}[Sketch of Proof]
Let $t:\mathbb{P}^1_k\xrightarrow[k]{\sim} \mathbb{P}^1_k$ be the isomorphism that satisfies $t(x_1)=0, t(x_2)=\infty$ and $t(x_3)=1$. (In particular,  $t(x_4)=\lambda(\varepsilon,\delta)$.)  Let  $H$ be an open subgroup of $\Pi^{(1,p')}_{\mathbb{P}^1_k-\Lambda}$ satisfying the conditions (i)(ii)(iii)(iv) in Definition \ref{2.1.1}. Then  $(U_{\overline{H}})^{\text{cpt}}\rightarrow X_{k^{\text{sep}}}$ is identified with $\mathbb{P}^1_{k^{\text{sep}}}\rightarrow\mathbb{P}^1_{k^{\text{sep}}}$, $x\mapsto x^{\ell^n}$, and the  corresponding extension of  function fields is $k^{\text{sep}}(t^{\frac{1}{\ell^{n}}})/k^{\text{sep}}(t)$. Let  $U_{H}\rightarrow U_{k(\mu_{\ell^{n}})}$ be the cover  corresponding to  $H \lhd p_{U/k}^{-1}(G_{k(\mu_{\ell^{n}})})$. Then there exists $\omega_H\in k(\mu_{\ell^{n}})^{\times}$ such that $K(U_{H})=k(\mu_{\ell^{n}},(\omega_{H}t)^{\frac{1}{\ell^{n}}})$ by the conditions (i)(ii)(iii)(iv) in Definition \ref{2.1.1} and Kummer theory.  Since the  fixed field $\kappa_H$ by  $\bigcap_{\tilde{y}} \text{pr}(H\cap D_{\tilde{y}})$  is the composite field of  residue fields of  all points of $X_{H}(:=(U_{H})^{\text{cpt}})$ above $x_3$ and $x_4$, we get $\kappa_H=k(\mu_{\ell^{n}},\left\{\omega_H\right\}^{\frac{1}{\ell^{n}}},\left\{\omega_H\lambda(\varepsilon,\delta)\right\}^{\frac{1}{\ell^{n}}})$.  There exists an  open subgroup $H_0$ of $\Pi^{(1,p')}_{\mathbb{P}^1_k-\Lambda}$ with  $\omega_{H_0}=1$ that satisfies the conditions (i)(ii)(iii)(iv) in Definition \ref{2.1.1}. Note that we have   $\kappa_H\supset\kappa_{H_0}$.  Hence, the assertion follows.
\end{proof}

\noindent\underline{\textbf{Step 2}}\par

First, we consider a generalization of  Proposition  \ref{step2prop}  to arbitrary  $p\geq 0$.   When $p>0$,  for a  finitely generated subgroups $\Gamma$ of  $k^{\times}$, we write $\Gamma^{\text{perf}}\subset \overline{k}^{\times}$ for the set of all elements whose  $p^{n}$-th power is in $\Gamma$ for some $n\in\mathbb{Z}_{\geq0}$. In other words, $\Gamma^{\text{perf}}:=\underset{n\in\mathbb{Z}_{\geq0}}\cup \Gamma^{\frac{1}{p^{n}}}$. \par

\begin{lmm}[cf. \cite{Na1990-405} (3.1)Lemma, \cite{Ya2020} Proposition 2.1.4.]\label{nakamuralmm2}
 Let $\Gamma_{1}$, $\Gamma_{2}$ be finitely generated subgroups of  $k^{\times}$. If  $k(\Gamma_{1}^{\frac{1}{\ell^{n}}})=k(\Gamma_{2}^{\frac{1}{\ell^{n}}})$ for all $n\in \mathbb{Z}_{\geq0}$ and  for all primes $\ell$ that differ from $p$, then the following hold.
\begin{enumerate}[(1)]
\item If $p=0$, then $\Gamma_{1}=\Gamma_{2}$.
\item If $p> 0$, then $\Gamma_{1}^{\text{\rm{perf}}}= \Gamma_{2}^{\text{\rm{perf}}}$. \par
\end{enumerate}
\end{lmm}

\begin{proof}[Sketch of Proof]
(1) is the  same as Proposition  \ref{step2prop}. We  consider (2). Assume that $p>0$.  As in the proof of  Proposition \ref{step2prop}, we define $R$ as   the integral closure of  $\mathbb{F}_{p}[\Gamma_{2}]$  in $k(\mu_{4})$.  We have that $R^{\times}$ is a finitely generated  $\mathbb{Z}$-module because  $R$  is  finitely generated over $\mathbb{F}_{p}$ by assumption. In the notation of Proposition \ref{step2prop},  we have $\gamma\in\Gamma_{1}\cdot R^{\times \ell^n}$  for all $n\in \mathbb{Z}_{\geq0}$ and  for all primes $\ell$ that differ from $p$. Thus, we obtain $\gamma\in\Gamma_{1}^{\text{perf}}$.
 \end{proof}

\begin{lmm}[cf. \cite{Ya2020} Lemma 2.2.1, \cite{Ya2020} Lemma 2.3.4]\label{2.2.1}
Let  $\lambda_{1},\lambda_{2}\in k^{\times}-\left\{1\right\}$, and let $\rho\in \overline{k}$ be a primitive $6$-th root of unity. 
\begin{enumerate}[(1)]
\item  If  $p=0$, $\langle\lambda_{1}\rangle =\langle \lambda_{2}\rangle$ and  $\langle 1-\lambda_{1}\rangle=\langle 1-\lambda_{2}\rangle$ in $k^\times$, then $\lambda_{1}=\lambda_{2}$ or $\left\{\lambda_{1},\lambda_{2}\right\}=\left\{\rho,\rho^{-1}\right\}$.
\item If  $p>0$, $\lambda_{1}\not\in k\cap \overline{\mathbb{F}}_p$, $\langle\lambda_{1}\rangle^{\text{\rm{perf}}}=\langle \lambda_{2}\rangle^{\text{\rm{perf}}} $ and $\langle1-\lambda_{1}\rangle^{\text{\rm{perf}}}=\langle 1-\lambda_{2}\rangle^{\text{\rm{perf}}}$ in $k^{\times}$, then there exists unique $n\in\mathbb{Z}$ such that $\lambda_{2}=\lambda_{1}^{p^n}$. 
\end{enumerate}
\end{lmm}

\begin{proof}[Sketch of Proof]
(1)  We fix an embedding  $k\hookrightarrow \mathbb{C}$ and regard  $k$ as a subfield of $\mathbb{C}$.  In particular, we may assume  that $\rho=\mathrm{\text{e}}^{\frac{\pi}{3}\sqrt{-1}}$.  If $|\lambda_{1}|\neq1$, then $\lambda_{1}$ is a non-torsion element and $\mathbb{Z}\cong\langle\lambda_{1}\rangle=\langle\lambda_{2}\rangle$, hence $\lambda_{1}\in \{\lambda_{2}, \lambda_{2}^{-1}\}$. Similarly,  if  $|1-\lambda_{1}|\neq1$, then $1-\lambda_{1}\in \{1-\lambda_{2}, (1-\lambda_{2})^{-1}\}$. Then we obtain  either  $\lambda_{1}=\lambda_{2}$ or  $|\lambda_{1}|=|\lambda_{2}|=|1-\lambda_{1}|=|1-\lambda_{2}|=1$ by  an  easy calculation.  The second case coincides with   $\left\{\lambda_{1},\lambda_{2}\right\}=\left\{\mathrm{e}^{\frac{\pi}{3}\sqrt{-1}},\mathrm{e}^{-\frac{\pi}{3}\sqrt{-1}}\right\}$.\par
(2) By assumption, $\lambda_{1}$, $1-\lambda_{1}$, $\lambda_{2}$ and $1-\lambda_{2}$ are   non-torsion elements of $k^{\times}$.  Since $\langle\lambda_{1}\rangle^{\text{\rm{perf}}}=\langle \lambda_{2}\rangle^{\text{\rm{perf}}} $ and $\langle1-\lambda_{1}\rangle^{\text{\rm{perf}}}=\langle 1-\lambda_{2}\rangle^{\text{\rm{perf}}}$, there exist a  unique pair  $u,v\in\mathbb{Z}$ such that  $\lambda_{1}^{p^u}\in \{\lambda_{2}, \lambda^{-1}_{2}\}$ and  $1-\lambda_{1}^{p^v}\in\{1-\lambda_{2}, (1-\lambda_{2})^{-1}\}$. We may assume   that $\lambda_{1}^{p^u}=\lambda_{2}^{-1}$ and $1-\lambda_{1}^{p^v}=(1-\lambda_{2})^{-1}$.  Hence $\lambda_{1}$ satisfies $\lambda_{1}^{p^{u+v}}-\lambda_{1}^{p^{v}}+1=0$.  Take  $W\in\mathbb{N}$ that  satisfies  $u+v+W\geq 0$ and $v+W\geq 0$. Then  $\lambda_{1}$ is a root of the polynomial $t^{p^{u+v+W}}-t^{p^{v+W}}+1\in \mathbb{F}_p[t]$. Hence we get $\lambda_{1}\in k\cap \overline{\mathbb{F}}_{p}$. This is absurd.\par

 \end{proof}

\noindent\underline{\textbf{Step 3}}\par
First, we consider the case of $p=0$. In this case, we obtain the $m$-step solvable Grothendieck conjecture for  hyperbolic punctured projective lines by Step $1$ and Step $2$.

\begin{prp}[cf. \cite{Na1990-405} (4.4)Theorem,  \cite{Ya2020} Proposition 2.2.4] \label{punctured0}
 Let $\Lambda_{1}$, $\Lambda_{2}$ be finite sets of $k$-rational points of $\mathbb{P}^1_{k}$ with $|\Lambda_{1}|\geq 3$.  Assume that $p=0$ and $m\geq 3$. Let  $\Phi_{m}: \Pi^{(m)}_{\mathbb{P}^1_k-\Lambda_{1}}\xrightarrow[G_k]{\sim} \Pi^{(m)}_{\mathbb{P}^1_k-\Lambda_{2}}$ be a $G_{k}$-isomorphism, and $\Phi_{1}: \Pi^{(1)}_{\mathbb{P}^1_k-\Lambda_{1}}\xrightarrow[G_k]{\sim} \Pi^{(1)}_{\mathbb{P}^1_k-\Lambda_{2}}$ the isomorphism induced by $\Phi_{m}$. Then  there exists $f\in \text{\rm{Aut}}_{k}(\mathbb{P}^1_k)$ such that
\begin{enumerate}[(a)]
\item $f(\Lambda_{1})=\Lambda_{2}$, and 
\item for each pair  $x_{1}\in \Lambda_{1}$ and  $x_{2}\in \Lambda_{2}$,  $f(x_{1})=x_{2} $ if and only if $\Phi_{1}(I_{x_{1},\overline{\Pi}^{1,p'}_{\mathbb{P}^1_k-\Lambda_{1}}})=I_{x_{2},\overline{\Pi}^{1,p'}_{\mathbb{P}^1_k-\Lambda_{2}}}$.
\end{enumerate}
\end{prp}

\begin{proof}[Sketch of Proof]
We may assume that $|\Lambda_{1}|\geq 4$. Let $x_1,x_2,x_3,x_4$ be distinct elements of $\Lambda_{1}$ and  $\varepsilon=\left\{x_1,x_2\right\},\delta=\left\{x_3,x_4\right\}$. Since $m\geq 3$, $\Phi_{1}$ preserves decomposition groups at cusps (Proposition \ref{1.4.5}).  Let $\alpha^{*}: \Lambda_{1}\xrightarrow{\sim}\Lambda_{2}$ be the unique bijection satisfying $\alpha^{*}(x_{1})=x_{2}\Leftrightarrow  \Phi_{1}(I_{x_{1},\overline{\Pi}^{1,p'}_{\mathbb{P}^1_k-\Lambda_{1}}})=I_{x_{2},\overline{\Pi}^{1,p'}_{\mathbb{P}^1_k-\Lambda_{2}}}$ ($x_{1}\in \Lambda_{1}$, $x_{2}\in \Lambda_{2}$).  By Proposition \ref{2.1.3} and  Lemma \ref{nakamuralmm2},  we get $ \langle \lambda(\varepsilon,\delta) \rangle=\langle \lambda(\alpha^*\varepsilon,\alpha^{*}\delta) \rangle$.  Set $\varepsilon'=\left\{x_3,x_2\right\},\delta'=\left\{x_1,x_4\right\}$. Then $\lambda(\varepsilon',\delta')=1-\lambda(\varepsilon,\delta)$. Hence we get $\langle1-\lambda(\varepsilon,\delta)\rangle=\langle1-\lambda(\alpha^*\varepsilon,\alpha^{*}\delta)\rangle$. Thus, by  Lemma \ref{2.2.1}(1), we obtain either  $\lambda(\varepsilon,\delta)=\lambda(\alpha^*\varepsilon,\alpha^{*}\delta)$ or $\left\{\lambda(\varepsilon,\delta),\ \lambda(\alpha^*\varepsilon,\alpha^{*}\delta)\right\}=\left\{\rho,\ \rho^{-1}\right\}$. There is no  isomorphism $\Pi^{(2, \text{pro-} 2)}_{\mathbb{P}^1_k-\left\{0,1,\infty,\rho\right\}}\xrightarrow[G_{k}]{\sim}\Pi^{(2,\text{pro-} 2)}_{\mathbb{P}^1_k-\left\{0,1,\infty,\rho^{-1}\right\}}$ which maps $I_{0,\Pi^{(2, \text{pro-} 2)}_{\mathbb{P}^1_k-\left\{0,1,\infty,\rho\right\}}}$, $I_{\infty,\Pi^{(2, \text{pro-} 2)}_{\mathbb{P}^1_k-\left\{0,1,\infty,\rho\right\}}}$, $I_{1,\Pi^{(2, \text{pro-} 2)}_{\mathbb{P}^1_k-\left\{0,1,\infty,\rho\right\}}}$ to $I_{0,\Pi^{(2, \text{pro-} 2)}_{\mathbb{P}^1_k-\left\{0,1,\infty,\rho^{-1}\right\}}}$, $I_{\infty,\Pi^{(2, \text{pro-} 2)}_{\mathbb{P}^1_k-\left\{0,1,\infty,\rho^{-1}\right\}}}$, $I_{1,\Pi^{(2, \text{pro-} 2)}_{\mathbb{P}^1_k-\left\{0,1,\infty,\rho^{-1}\right\}}}$, respectively (see \cite{Ya2020} Lemma 2.2.3(2)). This together with a certain improvement (\cite{Ya2020} Corollary 1.4.8(2)) of Proposition \ref{1.4.5} in the pro-$\ell$ setting,   we get $\lambda(\varepsilon,\delta)=\lambda(\alpha^*\varepsilon,\alpha^{*}\delta)$ for an arbitrary pair   $\varepsilon,\delta$. Therefore, the assertion follows.
 \end{proof}

Next, we consider the case of $p>0$. This case has  difficulties arising from the existence of  Frobenius twists. 

\begin{prp}[cf. \cite{Ya2020} Proposition 2.3.7]\label{puncturedp} Let $\Lambda_{1}$, $\Lambda_{2}$ be finite sets of $k$-rational points of $\mathbb{P}^1_{k}$ with $|\Lambda_{1}|\geq 3$.  Assume that $p>0$ and  $m\geq 3$. Let  $\Phi_{m}: \Pi^{(m,p')}_{\mathbb{P}^1_k-\Lambda_{1}}\xrightarrow[G_k]{\sim} \Pi^{(m,p')}_{\mathbb{P}^1_k-\Lambda_{2}}$ be a $G_{k}$-isomorphism, and  $\Phi_{1}^{(v_{1},v_{2})}: \Pi^{(1,p')}_{\mathbb{P}^1_k-\Lambda_{1}(v_{1})}\xrightarrow[G_k]{\sim} \Pi^{(1,p')}_{\mathbb{P}^1_k-\Lambda_{2}(v_{2})}$ the  isomorphism induced by $\Phi_{m}$ for each pair  $v_{1},v_{2}\in\mathbb{Z}_{\geq 0}$.  We assume that  the following condition: $(\dag)$ For each $S'\subset \Lambda_{1}\ \text{with}\ |S'|=4$,  the curve $ \mathbb{P}^{1}_{\overline{k}}-S'$ does not  descend to a curve over $\overline{\mathbb{F}}_{p}$.  Then there exist $w_{1},w_{2}\in \mathbb{Z}_{\geq 0}$ and  $f:\mathbb{P}^{1}_{k}\ (=\mathbb{P}^1_{k}(w_{1}))\xrightarrow[k]{\sim}\mathbb{P}^{1}_{k}\ (=\mathbb{P}^1_{k}(w_{2}))$ such that. 

\begin{enumerate}[(a)]
\item $f(\Lambda_{1}(w_{1}))=\Lambda_{2}(w_{2})$, and 
\item for each pair  $x_{1}\in \Lambda_{1}(w_{1})$ and  $x_{2}\in \Lambda_{2}(w_{2})$,  $f(x_{1})=x_{2} $ if and only if \\$\Phi_{1}^{(w_{1},w_{2})}(I_{x_{1},\overline{\Pi}^{1,p'}_{\mathbb{P}^1_k-\Lambda_{1}(w_{1})}})=I_{x_{2},\overline{\Pi}^{1,p'}_{\mathbb{P}^1_k-\Lambda_{2}(w_{2})}}$.
\end{enumerate}
\end{prp}
\begin{proof}[Sketch of Proof]
Let  $v_{1},v_{2}\in\mathbb{Z}_{\geq 0}$. Since  $\Phi_{1}^{(v_{1},v_{2})}$ preserves decomposition groups at cusps by Corollary \ref{1.4.5}(1), there exists a  bijection $\alpha_{1}^{*(v_{1},v_{2})}:\Lambda_{1}(v_{1})\rightarrow \Lambda_{2}(v_{2})$ such that $\alpha_{1}^{*(v_{1},v_{2})}(x_{1})=x_{2}\iff\Phi_{1}^{(v_{1},v_{2})}(I_{x_{1},\overline{\Pi}^{1,p'}_{\mathbb{P}^1_k-\Lambda_{1}(v_{1})}})=I_{x_{2},\overline{\Pi}^{1,p'}_{\mathbb{P}^1_k-\Lambda_{2}(v_{2})}}$  ($x_{1}\in \Lambda_{1}(v_{1})$, $x_{2}\in \Lambda_{2}(v_{2})$). For simplicity, we only consider the cases  that $|\Lambda_{1}|=4$ and  $|\Lambda_{1}|=5$. \par
Let  $x_1,x_2,x_3,x_4$ be distinct elements of $\Lambda_{1}$ and set  $\varepsilon=\left\{x_1,x_2\right\},\delta=\left\{x_3,x_4\right\}$.   Using  Proposition \ref{2.1.3} ,  Lemma \ref{nakamuralmm2} and   Lemma \ref{2.2.1}(2), there  exists a unique $n\in\mathbb{Z}$ such that $\lambda(\alpha_{1}^{*(0,0)}\varepsilon,\alpha_{1}^{*(0,0)}\delta)=\lambda(\varepsilon,\delta)^{p^n}$ by the same way as  the proof of Proposition \ref{punctured0}. (Remark: $\lambda(\varepsilon,\delta)$ is not contained in $k\cap\overline{\mathbb{F}}_p$ by the condition $(\dag)$.)  Thus, the assertion follows  when  $|\Lambda_{1}|=4$.\par

We consider the case that  $|\Lambda_{1}|=5$. Assume that $\Lambda_{1}=\left\{0,\infty,1, \lambda_{1,1}, \lambda_{1,2}\right\}$, $\Lambda_{2}=\left\{0,\infty,1,\lambda_{2,1}, \lambda_{2,2}\right\}$ and $\alpha_{1}^{*(0,0)}(0)=0$, $\alpha_{1}^{*(0,0)}(\infty)=\infty$, $\alpha_{1}^{*(0,0)}(1)=1$, $ \alpha_{1}^{*(0,0)}(\lambda_{1,i})=\lambda_{2,i}$ for $i=1,2$. Since $\lambda(\left\{0,\infty\right\},\left\{1,\lambda_{1,i}\right\})=\lambda_{1,i}$,  there exists a unique $n_i\in\mathbb{Z}$ such that $\lambda_{2,i}=\lambda_{1,i}^{p^{n_i}}\ (i=1,2)$. Similarly,  there exists a unique $\sigma,\tau,\zeta\in\mathbb{Z}$ such that $\frac{\lambda_{2,2}}{\lambda_{2,1}}=(\frac{\lambda_{1,2}}{\lambda_{1,1}})^{p^{\sigma}}$,$ \frac{\lambda_{2,2}-1}{\lambda_{2,1}-1}=(\frac{\lambda_{1,2}-1}{\lambda_{1,1}-1})^{p^{\tau}}$, $\frac{\lambda_{2,2}}{\lambda_{2,2}-1}\frac{\lambda_{2,1}-1}{\lambda_{2,1}}=(\frac{\lambda_{1,2}}{\lambda_{1,2}-1}\frac{\lambda_{1,1}-1}{\lambda_{1,1}})^{p^{\zeta}}$. Thus, by \cite{Ya2020} Lemma 2.3.5 and \cite{Ya2020} Lemma 2.3.6, we obtain  either $n_1=n_2$ or $\sigma=\tau=\zeta$.  When  $n_{1}=n_{2}$, the assertion follows.  Let  $T,T'\subset \Lambda_{1}$ with  $|T\cap T'|\geq 3$,  If  there exist  $w_{1},w_{2}\in\mathbb{Z}$ such that the conditions (a)(b) for ($w_{1},w_{2}, \mathbb{P}^{1}_{k}-T$) and  ($w_{1},w_{2}, \mathbb{P}^{1}_{k}-T'$) hold, then  the conditions (a)(b) for ($w_{1},w_{2}, \mathbb{P}^{1}_{k}-T\cup T'$) also  follows.   Thus, the assertion follows when $\sigma=\tau=\zeta$.

\end{proof}

Proposition \ref{punctured0} and Proposition \ref{puncturedp} are  stronger than the (weak) $m$-step solvable Grothendieck conjecture. In \cite{Na1990-411}, Nakamura uses such  functoriality and Galois descent to show the Grothendieck conjecture for genus $0$ hyperbolic curves. In \cite{Ya2020}, the author followed this method to show Theorem \ref{thmnyamaguchi}. 

\begin{proof}[Sketch of Proof of Theorem \ref{thmnyamaguchi}]
Let $\Phi_{m}: \Pi^{(m,p')}_{U_{1}}\xrightarrow[G_{k}]{\sim} \Pi^{(m,p')}_{U_{2}}$, and let $\Phi_{1}: \Pi^{(1,p')}_{U_{1}}\xrightarrow[G_{k}]{\sim}\Pi^{(1,p')}_{U_{2}}$ the isomorphism  induced by $\Phi_{m}$.  We may assume that  $g_{2}=0$ by \cite{Ya2020} Corollary 1.3.5. When $p=0$ (resp. $p>0$), there exists (resp. exist) a finite Galois extension $K$ of  $k$  (resp. and $M\in \mathbb{Z}_{\geq 0}$) such that  $U_{1,K}\cong\mathbb{P}^1_K-\Lambda_{1}$ and $U_{2,K}\cong\mathbb{P}^1_K-\Lambda_{2}$ (resp. $U_{1,K}(M)\cong\mathbb{P}^1_K-\Lambda_{1}$ and $U_{2,K}(M)\cong\mathbb{P}^1_K-\Lambda_{2}$)  for some  $\Lambda_{1},\Lambda_{2}\subset\mathbb{P}^1_K(K)$.  By Proposition \ref{punctured0} (resp.  Proposition \ref{puncturedp}),  there exists (resp. exist)  $f:  \mathbb{P}^{1}_{K}\xrightarrow[K]{\sim} \mathbb{P}^{1}_{K}$ (resp.  and $w_{1},w_{2}\in \mathbb{Z}_{\geq 0}$) such that the conditions (a)(b) in  Proposition \ref{punctured0} (resp. Proposition \ref{puncturedp}) hold. We define  $U_{1}^{f}$, $U_{2}^{f}$,  $E_{1}^{f}$, $E_{2}^{f}$, $\Phi_{1,K}^{f}$    as   $U_{1}$,$U_{2}$, $E_{1}$,$E_{2}$, $\Phi_{1,K}$ (resp.  $U_{1}(M+w_{1}), U_{2}(M+w_{2})$,  $E_{1}(M+w_{1}), E_{2}(M+w_{2})$, $\Pi^{(1,p')}_{U_{1,K}(M+w_{1})}\xrightarrow[G_{K}]{\sim}\Pi^{(1,p')}_{U_{2,K}(M+w_{1})}$), respectively. Hence  there exists $f:  \mathbb{P}^{1}_{K}\xrightarrow[K]{\sim} \mathbb{P}^{1}_{K}$ such that $f(E^{f}_{1,K})=E_{2,K}^{f}$, and $f(x_{1})=x_{2} \iff \Phi_{1,K}^{f}(I_{x_{1},\Pi^{(1,p')}_{U^{f}_{1,K}}})=I_{x_{2},\Pi^{(1,p')}_{U^{f}_{2,K}}}$ by above.\par
Let $\rho(U_{i}^{f})$ be the image of  $\rho\in \text{Gal}(K/k)$ by $\text{Gal}(K/k)\rightarrow \text{Aut}_{U_{i}^{f}}(U^{f}_{i,K})$  ($i=1,2$).  Let $\underline{\rho}( U_{1}^{f})$ be the inverse image of $\rho$ by $p_{U/k}:\Pi^{(1,p')}_{U_{1}^{f}}\twoheadrightarrow G_k$ and $\underline{\rho}( U_{2}^{f})$ the image of $\underline{\rho}( U_{1}^{f})$ by $\Pi^{(1,p')}_{U^{f}_{1}}\xrightarrow{\sim},\Pi^{(1,p')}_{U^{f}_{2,K}}$. We have $\Phi_{1,K}^{f}(I_{\rho( U_{1}^{f})(x_{1}),\overline{\Pi}^{1,p'}_{U^{f}_{1,K}}})=\underline{\rho}( U_{2}^{f})\cdot \Phi_{1,K}^{f}(I_{x_{1},\overline{\Pi}^{1,p'}_{U^{f}_{1,K}}})\cdot \underline{\rho}( U_{2}^{f})^{-1}=I_{\rho( U_{2}^{f})(x_{1}),\overline{\Pi}^{1,p'}_{U^{f}_{2,K}}}$ for all $x_{1}\in E_{1,K}^{f}$. Hence we get $f(\rho( U_{1}^{f})(x_{1}))=\rho( U_{2}^{f})(f(x_{1}))$. As $|E_{1,K}^{f}|\geq3$, it follows that $ f\circ\rho( U_{1}^{f})=\rho( U_{2}^{f})\circ  f$ for all $\rho$. By Galois descent, we obtain $ U_{1}^{f}\underset{k}\cong U_{2}^{f}$.\par

 \end{proof}

\subsection{Work in Progress}
We notice that there are several possible extensions to  Theorem \ref{thmnyamaguchi}. For example,  
can $m\geq 3$ in Theorem \ref{thmnyamaguchi} be replaced with $m\geq 2$?  This question is open for now. However, the author thinks that the answer to  this question would be  affirmative because  we have Theorem \ref{thmnakamura} (see the beginning of  subsection 3.2).    As other examples,  we consider  the following two extensions.
\begin{enumerate}[(i)]
\item  The Grothendieck conjecture is  also proved when $k$ is a finite field. Hence we can expect that the $m$-step solvable Grothendieck conjecture is also true when $k$ is a finite field. Unfortunately, by the condition ($\dag$), Theorem \ref{thmnyamaguchi} does not imply   the result of the case that  $k$ is a finite field.  
\item We can also expect  that the $m$-step solvable Grothendieck conjecture for (general, in other words, genus $\geq0$) hyperbolic curves is true.  To consider this problem, we need an approach different from   the method in section $3$.  
\end{enumerate}
For the present, the author is  trying to show the above two extensions (i)(ii)  by constructing the $m$-step solvable version of the methods in \cite{Ta1997}, \cite{Mo2007}.

\section*{Acknowledgments}
I would like to  acknowledge  Professor Akio Tamagawa (Research Institute for Mathematical Sciences, Kyoto University).  He gave me a great deal of advice on this survey.


\end{document}